\documentclass{tran-l}
\usepackage{amssymb}
\usepackage{mathrsfs}
\usepackage{tikz-cd}
\usepackage{hyperref}

\newtheorem{theorem}{Theorem}[section]

\theoremstyle{definition}
\newtheorem{definition}[theorem]{Definition}

\newtheorem{example}[theorem]{Example}
\newtheorem{corollary}[theorem]{Corollary}

\theoremstyle{remark}

\numberwithin{equation}{section}

\newcommand{\lPM}[1]{#1{\operatorname{\mathbf{-PM}}}}

\begin{document}
 
\title[Spectral Fundamental Group]{Spectral Homotopy and the Spectral Fundamental Group}

\author{Biswajit Mitra}
\address{Department of Mathematics, The University of Burdwan, Burdwan Rajbati, West Bengal 713104}
\email{bmitra@math.buruniv.ac.in}

\author{Sourav Koner}
\address{Department of Mathematics, The University of Burdwan, Burdwan Rajbati, West Bengal 713104}
\email{harakrishnaranusourav@gmail.com}


\begin{abstract}
In this paper, we introduce an algebraic-topological invariant for commutative pm-rings, termed the spectral fundamental group, which is denoted by $\pi_{k}^{alg}(A)$. This group is defined via homotopy classes of loops within the space of induced spectral maps, which are generated by the $k$-algebra endomorphism monoid of the ring. We establish foundational properties of this invariant, proving that $\pi_{k}^{alg}(A)$ is an abelian group that naturally respects direct products and admits natural morphisms with respect to fully invariant subrings. Further, we establish an explicit isomorphism between the spectral fundamental group of certain continuous function rings and the classical fundamental group of their associated topological mapping spaces. Finally, utilizing a generalized dual number construction, we present an explicit example of a pm-ring that cannot be embedded into any function ring over a field of characteristic zero, yet possesses a nontrivial spectral fundamental group. This demonstrates that $\pi_{k}^{alg}(A)$ captures homotopical dynamics that are intrinsically algebraic.
\end{abstract}

\maketitle

\section{Introduction}

For a commutative ring $A$ with unity, the prime spectrum $\operatorname{Spec}(A)$ and the maximal spectrum $\operatorname{Max}(A)$ provide a fundamental bridge between algebra and topology. In particular, for pm-rings—rings in which every prime ideal is contained in a unique maximal ideal—the canonical retraction

$$
\mu : \mathrm{Spec}(A) \to \mathrm{Max}(A)
$$ 
defined by assigning to each prime ideal the unique maximal ideal containing it, induces a natural and well-behaved topological structure on $Max(A)$
This provides a convenient framework for studying the interplay between algebraic properties of the ring and the associated topological structure of its prime and maximal spectra.

A central theme in both commutative algebra and topology is the study of invariants that encode structural information. The goal of this paper is to introduce and study such an invariant. Specifically, we consider the space of self-maps of $\operatorname{Max}(A)$ that are induced by $k$-algebra endomorphisms of $A$, and we investigate its homotopical structure. This leads to the definition of a new invariant, which we call the spectral fundamental group, denoted by $\pi_{k}^{alg}(A)$. By construction, this group encodes homotopy classes of loops in the space of induced spectral maps, thereby measuring the existence of nontrivial continuous one-parameter families of algebra endomorphisms.

This perspective places our work at the intersection of commutative algebra and homotopy theory. The construction is closely related to classical topology: in the case of rings of continuous functions, we show that $\pi_{k}^{alg}(A)$ recovers the fundamental group of a mapping space. More precisely, for a compact Hausdorff space $X$, we obtain a natural isomorphism
$$
\pi_{k}^{alg}(C(X)) \cong \pi_1(C(X,X), id_{X}),
$$
thereby situating our invariant within the well-studied topology of function spaces. Despite this connection, the spectral fundamental group is not merely a reformulation of known topological invariants. A key feature of our construction is that it is defined purely in terms of algebraic endomorphisms, and therefore remains meaningful even for rings that do not arise from classical function spaces. To demonstrate this, we construct explicit examples of pm-rings that cannot be embedded into any function ring over a field of characteristic zero, yet admit a nontrivial spectral fundamental group. These examples show that $\pi_{k}^{alg}(A)$ can detect genuinely algebraic phenomena that are not accessible through traditional topological models. The main contributions of this paper are as follows:

We define the spectral fundamental group $\pi_{k}^{alg}(A)$ using homotopy classes of loops in the space of induced spectral maps. We prove that this invariant carries a natural abelian group structure and behaves well with respect to standard constructions, including direct products and fully invariant retracts. We establish a precise connection with classical topology by identifying $\pi_{k}^{alg}(C(X))$ with the fundamental group of the mapping space $C(X,X)$. We construct explicit examples of pm-rings exhibiting nontrivial spectral fundamental groups beyond the realm of function rings, thereby demonstrating the intrinsically algebraic nature of the invariant.

The paper is organized as follows. In Section $2$, we review the necessary background on pm-rings and topological function spaces. In Section $3$, we introduce endomorphism monoids, induced spectral maps, and the definition of the spectral fundamental group. Section $4$ develops the general structure theory of this invariant, including product behavior. Finally, Section $5$ is devoted to constructions that illustrate the scope of the theory beyond classical function spaces. Through this work, we aim to initiate a systematic study of homotopical invariants arising from algebraic endomorphisms, and to highlight a new avenue for interaction between commutative algebra and topology.

\section{Preliminaries}\label{preli}

For a commutative ring $A$ with unity, let $\mathrm{Spec}(A)$ and $\mathrm{Max}(A)$ denote the prime and maximal spectra of $A$, respectively. Both $\mathrm{Spec}(A)$ and $\mathrm{Max}(A)$ are endowed with the Zariski topology. A commutative ring with unity is said to be a pm-ring if every prime ideal is contained in a unique maximal ideal. The result below is well-known and can be found, for instance, in \cite{MO71} and \cite{C82}.

\begin{theorem}\label{m1}
The following are equivalent. \\
$(a)$ $A$ is a pm-ring. \\
$(b)$ $\mathrm{Max}(A)$ is a retract of $\mathrm{Spec}(A)$. \\
$(c)$ Whenever $a + b = 1$ in $A$, there exist $r, s \in A$ such that $(1 - ar)(1 - bs) = 0$. 
\end{theorem}

For a pm-ring $A$, it follows from part $(c)$ of Theorem \eqref{m1} that $\mathrm{Max}(A)$ is a Hausdorff space. Moreover, the map 
\begin{equation}\label{e1}
\mu : \mathrm{Spec}(A) \longrightarrow \mathrm{Max}(A),   
\end{equation}
which assigns to each prime ideal $P$ the unique maximal ideal
$\mu(P)$ containing $P$, is the unique retraction of $\mathrm{Spec}(A)$ onto $\mathrm{Max}(A)$. In what follows, we consider spaces of continuous maps involving products of topological spaces and their associated function spaces. In particular, we make use of a fundamental result describing the relationship between continuous maps on product spaces and continuous maps into function spaces. For the reader’s convenience, we recall this result below, which can be found, for example, in \cite{B64}.

\begin{theorem}[Exponential Law]\label{m2}
Let $X$ be a locally compact Hausdorff space and let $Y$ and $Z$ be topological spaces. Then
\[
C(X \times Z, Y) \cong C\bigl(Z, C(X, Y)\bigr),
\]
where all function spaces are endowed with the compact-open topology.
\end{theorem}

\section{Endomorphisms and Induced Spectral Maps}

Throughout this paper, we assume $k$ is one of the following rings: $\mathbb{Z}$ or $\mathbb{R}$. Let $\lPM{k}$ denote the category whose objects are pm-rings equipped with a $k$-algebra structure, and whose morphisms are unital $k$-algebra homomorphisms. For $A \in \mathrm{Ob}(\lPM{k})$, we write
$$
E_k(A) := \mathrm{Hom}_{\lPM{k}}(A, A)
$$
for its endomorphism monoid, that is, under composition, $E_{k}(A)$ is a monoid with identity $id_A$. For $f \in E_{k}(A)$, the map $f_{\ast} : \mathrm{Spec}(A) \to \mathrm{Spec}(A)$, defined by $P \mapsto f^{-1}(P)$, is continuous. Consequently, the restriction of $f_{\ast}$ to $\mathrm{Max}(A)$ is also continuous. By  \eqref{e1}, since $\mu$ is the unique retraction, it follows immediately that $\mu_f := \mu \circ f_{\ast}\big|_{\mathrm{Max}(A)}$ is continuous. Adopting this terminology, we make the following definition.

\begin{definition}[Induced Spectral Maps]
For $A \in \mathrm{Ob}(\lPM{k})$, define
$$
X_{k}^{A}:= \{ \mu_f \mid f \in E_{k}(A) \} \quad\text{and}\quad Y:= C(\mathrm{Max}(A), \mathrm{Max}(A)).
$$
The elements of $X_{k}^{A}$ are called induced spectral maps. We endow the space $Y$ with the compact-open topology and $X_{k}^{A}$ with the induced subspace topology. 
\end{definition}

\subsection{Homotopy of Spectral Loops}

\begin{definition}[Homotopy of Induced Spectral Maps]
Two induced spectral maps $\mu_f$ and $\mu_g$ in $X_{k}^{A}$ are said to be homotopic, written $\mu_{f} \simeq \mu_{g}$, if there exists a continuous path 
$$
\alpha: [0,1] \rightarrow X_{k}^{A} 
$$ 
such that $\alpha(0) = \mu_f$ and $\alpha(1) = \mu_g$.
\end{definition}

By Theorem \eqref{m2}, this amounts to stating that there exists a continuous map 
$\widetilde{\alpha}: \mathrm{Max}(A) \times [0,1]
\rightarrow \mathrm{Max}(A)$ 
such that
$$
\widetilde{\alpha}(-, t) = \mu_{h_{t}} \text{ for some } h_{t} \in E_{k}(A) \text{ with } h_{0} = f \text{ and } h_{1} = g.
$$
In this case, we also say that $f$ and $g$ are spectrally homotopic.

\begin{definition}[Spectral Loop at $\mu_{id_{A}}$]
A spectral loop at $\mu_{id_{A}}$ is a continuous map
$$
\gamma : [0, 1] \rightarrow X_{k}^{A}
$$
such that $\gamma(0) = \gamma(1) = \mu_{id_{A}}$. 
\end{definition}

\begin{definition}[Homotopy of Spectral Loops at $\mu_{id_{A}}$]
Two spectral loops $\gamma$ and $\eta$ are said to be \emph{homotopic}
if there exists a continuous map
$$
F : [0,1] \times [0,1] \rightarrow X_{k}^{A}
$$
such that
\begin{enumerate}
    \item $F(-, 0) = \gamma$ and $F(-, 1) = \eta$,
    \item $F(0, s) = F(1,s) = \mu_{id_{A}}$
    for all $s \in [0,1]$.
\end{enumerate}
\end{definition}

\subsection{Spectral Fundamental Group}

\begin{definition}\label{de5}
The spectral fundamental group of $A$, written $\pi_{k}^{alg}(A)$, is the fundamental group $\pi_{1}(X_{k}^{A}, \mu_{id_{A}})$, that is,
$$
\pi_{k}^{alg}(A)
:=
\{\text{spectral loops at $\mu_{id_{A}}$}\}\big/\text{homotopy}.
$$
The group operation is induced by concatenation of loops, and the identity
element is the constant loop at $\mu_{id_{A}}$.
\end{definition}

\section{General Results}

\begin{theorem}\label{m3}
$\pi_{k}^{alg}(A)$ is an abelian group.
\end{theorem}
\begin{proof}
The space $X_{k}^{A}$ is a submonoid of $Y$ under composition. Since $\mathrm{Max}(A)$ is a compact Hausdorff space, composition is continuous in the compact-open topology, so $X_{k}^{A}$ is a topological monoid. Thus, $\pi_{k}^{alg}(A)$ is the fundamental group of this H-space based at its identity element $\mu_{id_{A}}$. By the Eckmann–Hilton argument, the fundamental group of any H-space at its identity is abelian.
\end{proof}

Recall that a $k$-subalgebra $D$ of $B$ is said to be fully invariant in $B$ if $f(D) \subseteq D$ for every $f \in E_{k}(B)$. Also, $D$ is said to be a retract of $B$ if there exists a morphism $r: B \rightarrow D$ such that $r|_{D}$ is the identity map on $D$.

\begin{theorem}\label{m4}
Let $h: A \rightarrow B$ be a morphism with the kernel $I$, and suppose that $h(A)$ is fully invariant in $B$. Then there exists a natural morphism $h^{\ast}$ from $\pi_{k}^{alg}(B)$ to $\pi_{k}^{alg}(A/I)$. Moreover, if $h$ is an isomorphism, then so is $h^{\ast}$.
\end{theorem}
\begin{proof}
Let $h(A) = D$. Since $A/I \cong D$, let $g: A/I \rightarrow D$ be the isomorphism and let $\theta: \mathrm{Max}(D) \rightarrow \mathrm{Max}(A/I)$ be the homeomorphism, induced by $h$ and $g$ respectively. Let $\delta: \mathrm{Max}(B) \rightarrow \mathrm{Max}(D)$ be a map which sends every maximal ideal $M$ of $B$ into the unique maximal ideal of $D$ containing $M \cap D$. By theorem $(1.6)$ in \cite{MO71}, $\delta$ is a continuous surjection. If $\sigma: \mathrm{Max}(B) \rightarrow \mathrm{Max}(A/I)$ is the map such that $\sigma = (\theta \circ \delta)$, then it is easy to see that $\sigma$ is also a continuous surjection. Suppose now that $f \in E_{k}(B)$ and $M_{0} \in \mathrm{Max}(B)$. Let $\mu_{f}(M_{0}) = M_{1}$, $\delta(M_{1}) = M_{2}$, and $\theta(M_{2}) = M_{3}$. Also, let $\delta(M_{0}) = M_{4}$, $\theta(M_{4}) = M_{5}$, and $\mu_{g^{-1} \circ f|_{D} \circ g}(M_{5}) = M_{6}$. This implies the following
\begin{enumerate}
    \item $f^{-1}(M_{0}) \subseteq M_{1}$, $M_{1} \cap D \subseteq M_{2}$, and $g^{-1}(M_{2}) = M_{3}$;
    \item $M_{0} \cap D \subseteq M_{4}$ and $g^{-1}(f^{-1}(M_{4})) \subseteq M_{6}$.
\end{enumerate}
From $(1)$, we conclude that 
$$
g^{-1}(f^{-1}(M_{0})) \cap A/I \subseteq g^{-1}(M_{1}) \cap A/I \subseteq g^{-1}(M_{2}) \subseteq M_{3}.
$$
As $D$ is fully invariant in $B$, therefore from $(2)$, we obtain that 
$$
g^{-1}(f^{-1}(M_{0})) \cap A/I \subseteq g^{-1}(f^{-1}(M_{4})) \subseteq M_{6}.
$$
Since the homomorphic image of a pm-ring is a pm-ring, therefore, we conclude that $M_{3} = M_{6}$, that is, for each $f \in E_{k}(B)$, the rectangular diagram in \eqref{d1} commutes.
\begin{equation}
\begin{tikzcd}\label{d1}
	{\mathrm{Max}(B)} && {\mathrm{Max}(B)} \\
	\\
	{\mathrm{Max}(A/I)} && {\mathrm{Max}(A/I)}
	\arrow["{\mu_{f}}", from=1-1, to=1-3]
	\arrow["\sigma"', from=1-1, to=3-1]
	\arrow["\sigma", from=1-3, to=3-3]
	\arrow["{\mu_{(g^{-1} \circ f|_{D} \circ g)}}"', from=3-1, to=3-3]
\end{tikzcd}
\end{equation}
Define a map $\psi: X_{k}^{B} \rightarrow X_{k}^{A/I}$ by $\mu_{f} \mapsto \mu_{g^{-1} \circ f|_{D} \circ g}$. Let $[K, U]$ be a subbasic open set containing $\mu_{g^{-1} \circ f|_{D} \circ g}$.

Consider the subbasic open set $[\sigma^{-1}(K), \sigma^{-1}(U)]$. By the rectangular diagram in \eqref{d1}, we conclude that $\mu_{f} \in [\sigma^{-1}(K), \sigma^{-1}(U)]$, and that $\mu_{g^{-1} \circ l|_{D} \circ g} \in [K, U]$ for every $\mu_{l} \in [\sigma^{-1}(K), \sigma^{-1}(U)]$. This shows that $\psi$ is continuous. Finally, if we define the map $h^{\ast}: \pi_{k}^{alg}(B) \rightarrow \pi_{k}^{alg}(A/I)$ given by $[\mu_{f}] \mapsto [\mu_{g^{-1} \circ f|_{D} \circ g}]$, then the continuity of $\psi$ proves that $h^{\ast}$ is a well-defined map. Moreover, it is evident that $h^{\ast}$ is a group homomorphism.

If $h$ is an isomorphism, then by considering $h^{-1}: B \to A$ and by applying the first part, we get a group homomorphism $(h^{-1})^{\ast} : \pi_{k}^{alg}(A) \to \pi_{k}^{alg}(B)$,
which serves as the inverse of $h^{\ast}$. This completes the proof.
\end{proof}  

\begin{theorem}\label{maharajji}
If $D$ is a fully invariant retract of $B$, then the natural morphism $i^{\ast}$ from $\pi_{k}^{alg}(B)$ to $\pi_{k}^{alg}(D)$ is surjective. Moreover, $\pi_{k}^{alg}(B) \cong \pi_{k}^{alg}(D) \times \mathrm{ker}(i^{\ast})$.
\end{theorem}
\begin{proof}
By Theorem \eqref{m4}, there is a natural morphism $i^{\ast}: \pi_{k}^{alg}(B) \rightarrow \pi_{k}^{alg}(D)$ give by $[\mu_{f}] \mapsto [\mu_{f|_{D}}]$. Suppose now that $[\mu_{l}] \in \pi_{k}^{alg}(D)$. Let $r: B \rightarrow D$ be a retraction. Since $[\mu_{l \circ r}] \in \pi_{k}^{alg}(B)$ and since $(l \circ r)|_{D} = l$, we conclude that $i^{\ast}$ is a surjective morphism. 

Suppose that $f \in E_{k}(D)$. Let $r_{\ast}: \mathrm{Spec}(D) \rightarrow \mathrm{Spec}(B)$ be the continuous map induced by the retraction $r$ and let $\theta_{r}: \mathrm{Max}(D) \rightarrow \mathrm{Max}(B)$ be the map given by $\theta_{r} = \mu \circ r_{\ast}|_{\mathrm{Max}(D)}$, where $\mu: \mathrm{Spec}(B) \rightarrow \mathrm{Max}(B)$ is the unique retraction \cite{MO71}. Also, let $i: D \hookrightarrow B$ be the inclusion. Suppose $M_{1} \in \mathrm{Max}(D)$ be any. Let $\mu_{f}(M_{1}) = M_{2}$, $\theta_{r}(M_{2}) = M_{3}$, $\theta_{r}(M_{1}) = M_{4}$, and $\mu_{(i \circ f \circ r)}(M_{4}) = M_{5}$. This implies the following:
\begin{enumerate}
    \item $f^{-1}(M_{1}) \subseteq M_{2}$ and $r^{-1}(M_{2}) \subseteq M_{3}$;
    \item $r^{-1}(M_{1}) \subseteq M_{4}$ and $r^{-1}(f^{-1}(i^{-1}(M_{4}))) \subseteq M_{5}$.
\end{enumerate}
From $(1)$, we conclude that 
$$
r^{-1}(f^{-1}(M_{1})) \subseteq r^{-1}(M_{2}) \subseteq M_{3}.
$$
From $(2)$, we conclude that 
$$
r^{-1}(f^{-1}(i^{-1}(r^{-1}(M_{1})))) \subseteq r^{-1}(f^{-1}(i^{-1}(M_{4}))) \subseteq M_{5}.
$$
Notice now that $f \circ r = r \circ i \circ f \circ r$. Thus, we get that $M_{3} = M_{5}$, that is, for each $f \in E_{k}(D)$, the rectangular diagram in \eqref{d7} commutes. 
\begin{equation}
\begin{tikzcd}\label{d7}
	{\mathrm{Max}(D)} && {\mathrm{Max}(D)} \\
	\\
	{\mathrm{Max}(B)} && {\mathrm{Max}(B)}
	\arrow["{\mu_{f}}", from=1-1, to=1-3]
	\arrow["\theta_{r}"', from=1-1, to=3-1]
	\arrow["\theta_{r}", from=1-3, to=3-3]
	\arrow["{\mu_{(i \circ f \circ r)}}"', from=3-1, to=3-3]
\end{tikzcd}
\end{equation}
Define a map $\psi: X_{k}^{D} \rightarrow X_{k}^{B}$ by $\mu_{f} \mapsto \mu_{(i \circ f \circ r)}$. If $[K, U]$ be a subbasic open set containing $\mu_{(i \circ f \circ r)}$, then it is easy to see that $[\theta_{r}^{-1}(K), \theta_{r}^{-1}(U)]$ is a subbasic open set containing $\mu_{f}$. Moreover, for each $\mu_{g} \in [\theta_{r}^{-1}(K), \theta_{r}^{-1}(U)]$, we have that $\mu_{(i \circ g \circ r)} \in [K, U]$. This shows that the map $\psi$ is continuous. Hence, the map $t: \pi_{k}^{alg}(D) \rightarrow \pi_{k}^{alg}(B)$ defined by $[\mu_{f}] \mapsto [\mu_{(i \circ f \circ r)}]$ is a well-defined group homomorphism. If $\mathbf{1}$ is the identity map on $\pi_{k}^{alg}(D)$, then observe that $i^{\ast} \circ t = \mathbf{1}$, that is, $t$ is a lift to the identity map $\mathbf{1}$. Thus, $\pi_{k}^{alg}(B) \cong \pi_{k}^{alg}(D) \times \mathrm{ker}(i^{\ast})$.
\end{proof}

The following result follows trivially from Theorem \eqref{m4} and Theorem \eqref{maharajji}, so the proof is omitted.

\begin{theorem}
Let $h: A \rightarrow B$ be a morphism with the kernel $I$, and suppose that $h(A)$ is a fully invariant retract of $B$. Then the natural morphism $h^{\ast}$ from $\pi_{k}^{alg}(B)$ to $\pi_{k}^{alg}(A/I)$ is surjective. Moreover, $\pi_{k}^{alg}(B) \cong \pi_{k}^{alg}(A/I) \times \mathrm{ker}(h^{\ast})$.
\end{theorem}

A central goal in commutative ring theory is understanding how invariants behave under standard ring constructions. The group $\pi_{k}^{alg}$ naturally respects the direct product. This structural correspondence is formalized in the following.

\begin{theorem}\label{m5}
$\pi_{k}^{alg}(A \times B) \cong \pi_{k}^{alg}(A) \times \pi_{k}^{alg}(B)$.
\end{theorem}
\begin{proof}
Let $\Pi_{1}: A \times B \rightarrow A$ and $\Pi_{2}: A \times B \rightarrow B$ be the projections. Observe first that for each $f \in E_{k}(A \times B)$, $\Pi_{1} \circ f = f_{1} \in E_{k}(A)$ and $\Pi_{2} \circ f = f_{2} \in E_{k}(B)$. Conversely, if $h_{1} \in E_{k}(A)$ and $h_{2} \in E_{k}(B)$, then $h = (h_{1}, h_{2}): A \times B \rightarrow A \times B$ defined by $(a, b) \mapsto (h_{1}(a), h_{2}(b))$ is a morphism. Thus, there exists a one-to-one correspondence between the sets $E_{k}(A \times B)$ and $E_{k}(A) \times E_{k}(B)$.

Let $f \in E_{k}(A \times B)$, $\Pi_{1}(f) = f_{1}$, and $\Pi_{2}(f) = f_{2}$. As mentioned earlier, the map $\mu_{f}: \mathrm{Max}(A \times B) \rightarrow \mathrm{Max}(A \times B)$, given by $M \times B \mapsto \mu_{f_{1}}(M) \times B$ and $A \times N \mapsto A \times \mu_{f_{2}}(N)$, is continuous. Let $\sigma: \mathrm{Max}(A \times B) \rightarrow \mathrm{Max}(A) \sqcup \mathrm{Max}(B)$ be the homeomorphism given by $M \times B \mapsto M$ and $A \times N \mapsto N$. It is straightforward to see that the map $\mu_{f_{1}} \bigsqcup \mu_{f_{2}}: \mathrm{Max}(A) \sqcup \mathrm{Max}(B) \rightarrow \mathrm{Max}(A) \sqcup \mathrm{Max}(B)$, given by $M \mapsto \mu_{f_{1}}(M)$ and $N \mapsto \mu_{f_{2}}(N)$, is continuous. Since $\sigma \circ \mu_{f} = (\mu_{f_{1}} \bigsqcup \mu_{f_{2}}) \circ \sigma$, therefore, we conclude that for each $f \in E_{k}(A \times B)$, the rectangular diagram in \eqref{d2} commutes.
\begin{equation}\label{d2}
\begin{tikzcd}
{\mathrm{Max}(A \times B)} && {\mathrm{Max}(A \times B)} \\
\\
{\mathrm{Max}(A) \sqcup \mathrm{Max}(B)} && {\mathrm{Max}(A) \sqcup \mathrm{Max}(B)}
\arrow["{\mu_{f}}", from=1-1, to=1-3]
\arrow["\sigma"', from=1-1, to=3-1]
\arrow["\sigma", from=1-3, to=3-3]
\arrow["{\mu_{f_{1}} \bigsqcup \mu_{f_{2}}}"', from=3-1, to=3-3]
\end{tikzcd}
\end{equation}
 Define a map $\psi: X_{k}^{A \times B} \rightarrow X_{k}^{A} \times X_{k}^{B}$ by $\psi(\mu_{f}) = (\mu_{f_{1}}, \mu_{f_{2}})$. Since the sets $E_{k}(A \times B)$ and $E_{k}(A) \times E_{k}(B)$ are in a one-to-one correspondence, we conclude that $\psi$ is bijective. Finally, if we consider $[K_{A}, U_{A}] \times [K_{B}, U_{B}]$, the subbasic open set containing $(\mu_{f_{1}}, \mu_{f_{2}})$, then by the diagram in \eqref{d2}, we conclude that $\mu_{f} \in [\sigma^{-1}(K_{A} \cup K_{B}), \sigma^{-1}(U_{A} \cup U_{B})]$. This simultaneously shows that $\psi$ is both a continuous and an open map with $\psi(\mu_{id_{A \times B}}) = (\mu_{id_{A}}, \mu_{id_{B}})$.
\end{proof}

\subsection{Triviality of \texorpdfstring{$\pi_{k}^{alg}(A)$}{pi}}

If $A$ is a clean ring and if $\gamma:[0,1] \to X_{k}^{A}$ is a spectral loop, then $\gamma$ is constant, since $\mathrm{Max}(A)$ is totally disconnected, and so $\pi_{k}^{alg}(A) = 0$. Thus, for a topological space $X$, if $A$ is either of the rings $C_c(X)$, $C_{c}^{\ast}(X)$ or $C^{F}(X)$, then we can conclude that $\pi_{k}^{alg}(A) = 0$. 

For a given ring $A$, the group $\pi_{k}^{alg}(A)$ consists of homotopy classes of loops in the space of self-maps of $\mathrm{Max}(A)$ that arise from morphisms of $A$. Since every intermediate stage of a spectral homotopy must remain in $X_{k}^{A}$, such loops are constrained by the rigidity of morphisms. So, $\pi_{k}^{alg}(A)$ measures the extent to which $A$ admits continuous one-parameter families of morphisms. For spectrally rigid rings, the spectral fundamental group is trivial. It becomes nontrivial precisely when morphisms produce non-null-homotopic loops in the space of induced spectral maps. In this way, $\pi_{k}^{alg}(A)$ records homotopical dynamics intrinsic to the algebraic structure of $A$. 

\subsection{Nontriviality of \texorpdfstring{$\pi_{k}^{alg}(A)$}{pi}}

\begin{theorem}\label{m6}
Let $X$ be a topological space. Then for some Tychonoff space $Y,$ 
$$
\pi_{\mathbb{R}}^{alg}\!\big(C^{\ast}(X)\big)
\;\cong\;
\pi_{\mathbb{Z}}^{alg}\!\big(C^{\ast}(X)\big)
\;\cong\;
\pi_{1}\!\big(C(\beta Y,\beta Y), id_{\beta Y}\big).
$$
\end{theorem}
\begin{proof}
By theorem $(3.9)$ in \cite{GJ60}, there exists a Tychonoff space $Y$ such that $C(X) \cong C(Y)$; moreover, this isomorphism maps $C^{\ast}(X)$ onto $C^{\ast}(Y)$. Thus, applying Theorem \eqref{m4}, we get that $\pi_{\mathbb{Z}}^{alg}(C^{\ast}(X)) \cong \pi_{\mathbb{Z}}^{alg}(C^{\ast}(Y))$. Further, since $C^{\ast}(Y) \cong C(\beta Y)$, therefore again applying Theorem \eqref{m4}, we get that $\pi_{\mathbb{Z}}^{alg}(C^{\ast}(Y)) \cong \pi_{\mathbb{Z}}^{alg}(C(\beta Y))$. Let $A = C(\beta Y)$. Since $\mathrm{Max}(A) \cong \beta Y$, let $\theta: \mathrm{Max}(A) \rightarrow \beta Y$ be the homeomorphism given by $M_{p} \mapsto p$.

By theorem $(10.6)$ in \cite{GJ60}, each $f \in E(A)$ is induced by a unique continuous map $\phi_{f}: \beta Y \rightarrow \beta Y$. If $\mu_{f}: \mathrm{Max}(A) \rightarrow \mathrm{Max}(A)$ is the continuous map induced by $f \in E_{\mathbb{Z}}(A)$, then $\mu_{f}(M_{p}) = M_{\phi_{f}(p)}$. This amounts to stating that the rectangular diagram in \eqref{d3} commutes.
\begin{equation}
\begin{tikzcd}\label{d3}
	{\mathrm{Max}(A)} && {\mathrm{Max}(A)} \\
	\\
	{\beta Y} && {\beta Y}
	\arrow["{\mu_{f}}", from=1-1, to=1-3]
	\arrow["\theta"', from=1-1, to=3-1]
	\arrow["\theta", from=1-3, to=3-3]
	\arrow["{\phi_{f}}"', from=3-1, to=3-3]
\end{tikzcd}
\end{equation}
We now show that $X_{\mathbb{Z}}^{A} \cong C(\beta Y, \beta Y)$. Consider the map $\chi: X_{\mathbb{Z}}^{A} \rightarrow C(\beta Y, \beta Y)$ given by $\chi(\mu_{f}) = \phi_{f}$ for all $f \in E_{\mathbb{Z}}(A)$. It is evident that $\chi$ is bijective. Now, for $f \in E_{\mathbb{Z}}(A)$, suppose that $\phi_{f} \in [K, U]$ where $K$ is compact and $U$ is open in $\beta Y$. If we consider the subbasic open set $[\theta^{-1}(K), \theta^{-1}(U)]$, then by the diagram in \eqref{d3}, $\mu_{f} \in [\theta^{-1}(K), \theta^{-1}(U)]$ and $\chi([\theta^{-1}(K), \theta^{-1}(U)]) = [K, U]$. This simultaneously shows that $\chi$ is both a continuous and an open map with $\chi(\mu_{id_{A}}) = id_{\beta Y}$. 

Finally, since every unital ring homomorphism from $C(X)$ to itself is an $\mathbb{R}$-algebra homomorphism, we conclude that $\pi_{\mathbb{R}}^{alg}\!\big(C^{\ast}(X)\big) \cong \pi_{\mathbb{Z}}^{alg}\!\big(C^{\ast}(X)\big)$.
\end{proof}

\begin{corollary}\label{d0}
If $X$ is a compact Hausdorff space, then 
$$
\pi_{\mathbb{R}}^{alg}\!\big(C(X)\big)
\;\cong\;
\pi_{\mathbb{Z}}^{alg}\!\big(C(X)\big)
\;\cong\;
\pi_{1}\!\big(C(X, X), id_{X}\big).
$$
\end{corollary}
\begin{proof}
In this case, $C(X) \cong C^{\ast}(X)$ and $\beta X \cong X$. Thus, applying Theorem \eqref{m6}, the result follows.
\end{proof}

\begin{example}
Let $A = C(S^1)$, where $S^1$ denotes the unit circle. By Corollary \eqref{d0}, we have
$$
\pi_{\mathbb{Z}}^{alg}(A) \cong \pi_{1}(C(S^1, S^1), id_{S^1}).
$$
It is well-known that $\pi_{1}(C(S^1, S^1), id_{S^1}) \cong \mathbb{Z}$.
Hence, $\pi_{k}^{alg}(C(S^1)) \cong \mathbb{Z}$.  
\end{example}

\section{Spectral fundamental groups beyond function spaces}

Throughout this section, we work with the category $\lPM{\mathbb{R}}$. Let $T$ be a path-connected compact Hausdorff space. Motivated by Corollary \eqref{d0}, it is natural to ask the following questions: First, does there exist a pm-ring $A$ that cannot be realized as a subring of a function ring over a field of characteristic zero and such that $\mathrm{Max}(A) \cong T$? Second, does there exist a pm-ring $A$ that cannot be realized as a subring of a function ring over a field of characteristic zero, yet possesses a nontrivial spectral fundamental group? 

One might initially expect that such rings can only be realized as subrings of a function ring over a field of characteristic zero. However, our first theorem utilizes a generalized dual number construction to provide an explicit example of a pm-ring whose maximal spectrum is homeomorphic to a given topological space, but which fails to embed algebraically into any function ring over a field of characteristic zero. Furthermore, our second theorem provides an explicit example of an algebraic pm-ring with a nontrivial spectral fundamental group.

Let $C(T)$ be the ring of all real-valued continuous functions on $T$ and let $B = C(T)$. Let $\Delta$ be any set. For each $\delta \in \Delta$ let $n_{\delta} \in \mathbb{N}$ be any such that $n_{\delta} > 1$. Throughout this section, we work with the ring $A = B[x_{\delta} : \delta \in \Delta]/(x_{\delta}^{n_{\delta}}: \delta \in \Delta)$.

\begin{theorem}\label{m8}
$A$ cannot be embedded into any function ring over a field of characteristic zero, yet $\mathrm{Max}(A)$ is a path-connected compact Hausdorff space with $|\mathrm{Max}(A)| > 1$.
\end{theorem}
\begin{proof}
Since $A$ is not reduced, so it is not isomorphic to any subring of a function ring over a field of characteristic zero.

Let us consider the ideal $I = (x_{\delta}: \delta \in \Delta)/(x_{\delta}^{n_{\delta}}: \delta \in \Delta)$ of $A$. Since $A/I \cong B$, there is a surjective ring homomorphism $\varphi: A \rightarrow B$ with kernel $I$. Observe that $\mathrm{N}(A) = I$. As $I \subseteq \mathrm{J}(A)$, we conclude that $\varphi^{-1}(\mathrm{J}(B)) = \mathrm{J}(A)$. Also, since $\mathrm{J}(B) = 0$, we obtain that $\mathrm{J}(A) = I$. The map $\varphi_{\ast}: \mathrm{Max}(B) \rightarrow \mathrm{Max}(A)$ is a homeomorphism is well-known, since $\varphi$ is surjective and $V(I) = \mathrm{Max}(A)$. Now, if we take $T$ as a path-connected compact Hausdorff space with $|T| > 1$, then $\mathrm{Max}(A) \cong T \cong \mathrm{Max}(B)$. Further, since $\mathrm{Max}(A)$ is Hausdorff and since $\mathrm{J}(A) = \mathrm{N}(A)$, we conclude that $A$ is a pm-ring (see \cite{C82}).
\end{proof}

\begin{theorem}
If $\pi_{\mathbb{R}}^{alg}(B) \neq 0$, then $\pi_{\mathbb{R}}^{alg}(A) \neq 0$.
\end{theorem}
\begin{proof}
There exists a unique morphism $\tilde{r}: B[x_\delta : \delta \in \Delta] \to B$ that extends the identity map on $B$ and satisfies $\tilde{r}(x_\delta) = 0$ for all $\delta \in \Delta$. Let $J = (x_{\delta}^{n_{\delta}}: \delta \in \Delta)$. For any $\delta \in \Delta$, we have $\tilde{r}(x_\delta^{n_\delta}) = (\tilde{r}(x_\delta))^{n_\delta} = 0$. Thus, $\tilde{r}(J) = \{0\}$. This shows that $\tilde{r}$ factors through the canonical projection $\pi: B[x_\delta : \delta \in \Delta] \to A$, inducing a well-defined morphism $r: A \to B$ such that $r \circ \pi = \tilde{r}$. For any $b \in B$, viewed as an element of $A$, we have $r(b) = \tilde{r}(b) = b$. Therefore, $r|_B$ is the identity map on $B$, proving that $B$ is a retract of $A$.

Let $\sigma: A \to A$ be an arbitrary morphism. Then $A$ possesses a natural grading by total degree, that is, 
$$
A = \bigoplus_{j = 0}^{\infty} A_{j},
$$
where $A_{0} = B$, and for $j \geq 1$, $A_{j}$ is the free $B$-module generated by the set
$$
M_{j} = \bigg\{ \prod \bar{x}_{\delta}^{\alpha_{\delta}} : \sum \alpha_{\delta} = j, \, 0 \leq \alpha_{\delta} < n_{\delta}\bigg\}.
$$
Assume, for the sake of contradiction, that $\sigma(B) \not\subseteq B$. This means there exists some $b \in B$ such that $\sigma(b)$ has a non-zero component in $A_{j}$ for some $j \geq 1$. Let $k \geq 1$ be the minimal integer such that there exists some $b \in B$ where the degree $k$ component of $\sigma(b)$ is non-zero. By the minimality of $k$, for any $c \in B$, the projection of $\sigma(c)$ onto $A_j$ is identically zero for all $0 < j < k$. Thus, for any $c \in B$, we can write $\sigma(c) = \sigma_{0}(c) + \sigma_{k}(c) + R(c)$, where $\sigma_{0}(c) \in B$, $\sigma_{k}(c) \in A_{k}$, and $R(c) \in \bigoplus_{j>k} A_{j}$. Since $\sigma$ is a morphism, for any $b, c \in B$, we have $\sigma(bc) = \sigma(b)\sigma(c)$. Expanding both sides yields the following:
\begin{equation*}
\sigma_{0}(bc) + \sigma_{k}(bc) + R(bc) = \big(\sigma_{0}(b) + \sigma_{k}(b) + R(b)\big)\big(\sigma_{0}(c) + \sigma_{k}(c) + R(c)\big).
\end{equation*}
Multiplying the right side and collecting terms by degree, we see the degree zero term is $\sigma_{0}(b)\sigma_{0}(c)$. As $\sigma(b + c) = \sigma(b) + \sigma(c)$, therefore we obtain that $\sigma_{0}(b + c) = \sigma_{0}(b) + \sigma_{0}(c)$. Hence, $\sigma_{0}: B \rightarrow B$ is a morphism. Further, the degree $k$ term is $\sigma_{0}(b)\sigma_{k}(c) + \sigma_{0}(c)\sigma_{k}(b)$. Equating the degree $k$ components from both sides of the equation gives:
\begin{equation} \label{eq:derivation}
\sigma_{k}(bc) = \sigma_{0}(b)\sigma_{k}(c) + \sigma_{0}(c)\sigma_{k}(b).
\end{equation}
Since $A_{k}$ is a free $B$-module with basis $M_{k}$, therefore $\sigma_{k}(b)$ can be expressed uniquely in terms of this basis $\sigma_{k}(b) = \sum_{m \in M_{k}} D_{m}(b) m$, where each coefficient map $D_{m}: B \to B$ inherits the identity from \eqref{eq:derivation}:
$$
D_{m}(bc) = \sigma_{0}(b)D_{m}(c) + \sigma_{0}(c)D_{m}(b).
$$
As $\sigma$ is $\mathbb{R}$-linear, we conclude that each coefficient map $D_{m}$ is $\mathbb{R}$-linear. The fact that $D_{m} = 0$ is a natural extension of the standard result that commutative rings in which elements have roots admit no non-trivial derivations. For the sake of completeness, we include a self-contained proof, which relies only on the basic algebraic properties of $B$.  

Since $D_{m}(1) = 0$, therefore $D_{m}(\beta) = 0$ for all $\beta \in \mathbb{R}$. Let $g \in B$ and let $y \in T$. Also, let $\sigma_{0}(g)(y) = t$ and let $F = g - t$. It is easy to see that $\sigma_{0}(F)(y) = 0$. If we take $G = F^{\frac{1}{3}}$ and $H = F^{\frac{2}{3}}$, then $F = GH$ with $\sigma_{0}(G)(y) = 0$ and $\sigma_{0}(H)(y) = 0$. As $D_{m}(F) = \sigma_{0}(G)D_{m}(H) + \sigma_{0}(H)D_{m}(G)$, we obtain that $D_{m}(F)(y) = 0$. Further, since $D_{m}(g) = D_{m}(F)$, we conclude that $D_{m}(g)(y) = 0$. Since $y \in T$ was arbitrary, we get that $D_{m} = 0$. Because $D_{m} = 0$ for all $m \in M_{k}$, we conclude that $\sigma_{k}(b) = 0$ for all $b \in B$. This contradicts the assumption that $k$ is the minimal positive integer for which $\sigma_{k}$ is non-zero. This shows that $\sigma(b) = \sigma_{0}(b) \in B$ for all $b \in B$. Hence, $B$ is fully invariant in $A$.

Now, Theorem \eqref{maharajji} implies that $i^{\ast}: \pi_{\mathbb{R}}^{alg}(A) \rightarrow \pi_{\mathbb{R}}^{alg}(B)$ is surjective, where $i: B \hookrightarrow A$ is the inclusion. Since $\pi_{\mathbb{R}}^{alg}(B) \neq 0$, we conclude that $\pi_{\mathbb{R}}^{alg}(A) \neq 0$.
\end{proof}

\bibliographystyle{amsplain}

\end{document}